\documentclass{article}

\usepackage{amsmath,amssymb,amsthm,mathrsfs}
\usepackage{graphicx}

						  \newtheorem{theorem}{Theorem}
                          \newtheorem{corollary}[theorem]{Corollary}
                          
                          \newtheorem{proposition}[theorem]{Proposition}
\theoremstyle{definition} \newtheorem*{definition}{Definition}
\theoremstyle{definition} \newtheorem{hypothesis}{Hypothesis}
\theoremstyle{definition} \newtheorem{remark}{Remark}
\theoremstyle{definition} \newtheorem{example}{Example}

\begin{document}

\title{Continuous Dependence on Coefficients for Stochastic Evolution Equations with Multiplicative L\'evy Noise and Monotone Nonlinearity}
\author{Erfan Salavati, Bijan Z. Zangeneh\\ \\
Department of Mathematical Sciences\\
Sharif University of Technology\\
Tehran, Iran}
\date{}

\maketitle

\begin{abstract}
Semilinear stochastic evolution equations with multiplicative L\'evy noise and monotone nonlinear drift are considered. Unlike other similar works, we do not impose coercivity conditions on coefficients. We establish the continuous dependence of the mild solution with respect to initial conditions and also on coefficients which as far as we know, has not been proved before. As corollaries of the continuity result, we derive sufficient conditions for asymptotic stability of the solutions, we show that Yosida approximations converge to the solution and we prove that solutions have Markov property. Examples on stochastic partial differential equations and stochastic delay differential equations are provided to demonstrate the theory developed. The main tool in our study is an inequality which gives a pathwise bound for the norm of stochastic convolution integrals.
\end{abstract}

\section{Introduction}\label{section: introduction}

\subsection{Motivation}

Stochastic evolution equations have been an active area of research for many years. In the simplest case these equations are of the form
\begin{equation*}
    dX_t=AX_t dt + f(X_t) dt + g(X_t)d W_t
\end{equation*}
in a Hilbert space where $A$ is the infinitesimal generator of a $C_0$ semigroup of linear operators, $W_t$ is a Wiener process or more generally a martingale and $f$ and $g$ are assumed to be Lipschitz. Among studies with these assumptions one can note Da Prato and Zabczyk~\cite{DaPrato_Zabczyk_book}, in which the existence and uniqueness of the mild solution for stochastic evolution equations with Wiener noise is proved, as well as Kotelenez~\cite{Kotelenez-1984} in which the general martingale noise is considered.

The extensions of these results to the case of more general nonlinearity (non-Lipschitz) $f$, have been the subject of many papers. There are two main approaches in the study of non-Lipschitz stochastic evolution equations. First approach considers equations of the type
    \[ dX_t = F(X_t)dt + G(X_t)dW_t \]
in a Hilbert space $H$ equipped with a Banach space $B$ with dense embeddings  $B \subset H \subset B^*$, where $W_t$ is a Wiener process with values in a Hilbert space and $F$ and $G$ are generally assumed to be unbounded nonlinear operators that satisfy certain monotonicity and coercivity properties. This approach is called the variational method. For this approach see~\cite{Pardoux},~\cite{Krylov-Rozovskii} and~\cite{Rockner} for Wiener noise,~\cite{Gyongy} for general martingales and~\cite{Brzezniak-Liu-Zhu} for L\'evy noise.

The second approach is the semigroup approach to semilinear stochastic evolution equations with monotone drift, and considers equations of the form
\begin{equation}\label{equation: wiener_noise}
    dX_t=AX_t dt + f(X_t) dt + g(X_t)d W_t,
\end{equation}
where $W_t$ is a Wiener process and $f$ has a monotonicity assumption, i.e. there exists a real constant $M$ such that $\langle f(x)-f(y) , x-y \rangle \le M \|x-y\|^2$.

This approach has first appeared in deterministic context in the works of Browder~\cite{Browder} and Kato~\cite{Kato} and has been extended to stochastic evolution equations in~\cite{Zangeneh-Thesis} and~\cite{Zangeneh-Paper}.

Monotone operators are also called \emph{dissipative} operators in the literature and they are generalizations of decreasing real functions. Every operator of the form $f=g+h$ where $g$ is monotone and $h$ is Lipschitz, is a semimonotone operator and vice versa. Hence this approach is a generalization of the Lipschitz case. This generalization is useful since there are natural semimonotone functions which are not Lipschitz; examples include decreasing real functions, such as $-\sqrt[3]{x}$, or the sum of a non differentiable decreasing function with a Lipschitz function. Figure~\ref{figure:semimonotone} shows a semimonotone real function.

\begin{figure}[ht] \label{figure:semimonotone}
		\centering
		\includegraphics[scale=.5]{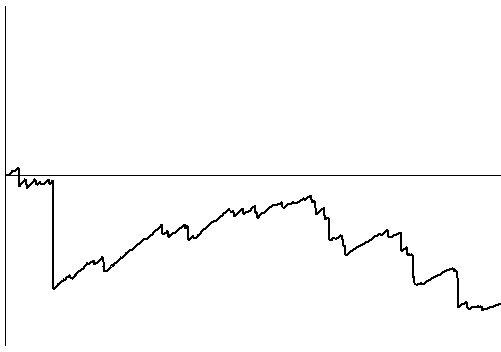}
		\caption{A semimonotone function}
\end{figure}

The semigroup approach to semilinear stochastic evolution equations with monotone nonlinearities has an advantage relative to the variational method since it does not require the coercivity. There are important examples, such as stochastic partial differential equations of hyperbolic type with monotone nonlinear terms, for which the generator does not satisfy the coercivity property and hence the variational method is not directly applicable to these equations. Pardoux~\cite{Pardoux} has developed a new theory for the application of the variational method to second order hyperbolic equations. But as is shown in Examples~\ref{example: finite_diemnsional_noise_hyperbolic} and~\ref{example:general_parabolic}, this problem can be treated directly in semigroup setting. Another advantage of the semigroup approach to semilinear stochastic evolution equations with monotone nonlinearities is that it allows a unified treatment of different problems, such as stochastic partial differential equations of hyperbolic and parabolic types and stochastic delay differential equations.

There are other works with this approach, e.g the exponential asymptotic stability of solutions in the case of Wiener noise has been studied in~\cite{Jahanipour-Zangeneh}, stochastic delay evolution equations has been studied in~\cite{Jahanipour-delay}, generalizing the previous results to stochastic functional evolution equations with coefficients depending on the past path of the solution is done in~\cite{Jahanipur-functional-stability}, a stopped version of~\eqref{main_equation} in case of Wiener noise has been studied in~\cite{Hamedani-Zangeneh-existence}, the large deviation principle for the case of Wiener noise is studied in~\cite{Dadashi-Zangeneh}. A limiting problem of such equations arising from random motion of highly elastic strings has been considered in~\cite{Zamani-Zangeneh-random-motion}. Finally, the stationarity of a mild solution to a stochastic evolution equation with a monotone nonlinear drift and Wiener noise is studied in~\cite{Zangeneh_Nualart}.

While the literature for equations with continuous noise is quite rich, not many works are done about equations with jump noise. In recent years some research has appeared on stochastic evolution equations with L\'evy (jump) noise, see e.g. Peszat and Zabczyk~\cite{Peszat-Zabczyk}, Albeverio, Mandrekar and R\"udiger~\cite{Albeverio-Mandrekar-Rudiger-2009} and Marinelli, Pr\'ev\^ot and R\"ockner~\cite{Marinelli-Prevot-Rockner} for the case of Lipschitz coefficients and Brze\'zniak, Liu and Zhu~\cite{Brzezniak-Liu-Zhu} for coercive and monotone coefficients with variational method. There are a number of works that have considered monotone (dissipative) coefficients with additive L\'evy noise, see e.g Peszat and Zabczyk~\cite{Peszat-Zabczyk}.

We should mention the article by Marinelli and R\"ockner~\cite{Marinelli-Rockner-wellposedness} which considers monotone nonlinear drift and multiplicative Poisson noise on certain function spaces and proves the existence, uniqueness and regular dependence of the mild solution on initial data. They impose an additional positivity assumption on the semigroup and the drift term is the Nemitsky operator associated with a real monotone function. Their idea is to regularize the monotone nonlinearity $f$ by its Yosida approximation $f_\lambda(x)= \lambda^{-1}(x-(I+\lambda f)^{-1}(x))$. We will treat their result as a special case of our theory in Example~\ref{example: finite_diemnsional_noise}.

In this article we have relaxed these assumptions, the semigroup is any exponentially growing $C_0$-semigroup on any seperable Hilbert space and the drift term is a general semimonotone operator. Our method is completely different.

The main contribution of this article is Theorem~\ref{theorem: continuity I} in section~\ref{section: Continuity With Respect to Parameter} which shows the continuous dependence of the solution of~\eqref{main_equation} on initial conditions and coefficients which as far as we know, has not been proved before in the literature. The problem of continuity, apart of having its own intrinsic interest, is also motivated by several other considerations, such as the study of the stability of models based on stochastic partial differential equations (SPDEs) and the convergence of numerical approximation schemes (\cite{Marinelli-DiPersio-Ziglio}). We mention below some other works in the literature about continuous dependence. In the context of Wiener noise,~\cite{DaPrato_Zabczyk_paper} considers the case that the semigroup is analytic and $f$ is locally Lipschitz, and shows that the solution is a continuous function of the noise coefficients.~\cite{Zangeneh-Thesis} generalizes this result to stochastic evolution equations with Wiener noise and monotone nonlinearity and shows that the solution depends continuously on initial condition and coefficients (including $A$). In the context of Poisson noise,~\cite{Albeverio-Mandrekar-Rudiger-2009} proves the continuous dependence on initial data and coeffincients for the case of Lipshcitz coefficients and~\cite{Marinelli-Prevot-Rockner} proves continuous dependence on initial data and under additional assumptions proves G\^ateaux and Fr\'echet differentiability of the solution w.r.t initial data in the case of Lipschitz coefficients. We should mention the recent article~\cite{Marinelli-DiPersio-Ziglio} in which is shown the continuous dependence of the solution of~\eqref{main_equation} on all coefficients including $A$ for the Lipschitz case. For equations with monotone coefficients,~\cite{Marinelli-Rockner-wellposedness} proves the continuous dependence of the solution w.r.t initial condition.~\cite{Jahanipur-boundedness} proves the continuous dependence on coefficients in the case of Wiener noise with monotone nonlinearity.

In the remainder of the introductory section we provide some preliminaries which will be useful in next sections.

\subsection{Stochastic Integration}

The main equation that we wish to study, consists of two noise terms, which are integrals with respect to a Wiener process and a compensated Poisson random measure. For the Wiener process term we use cylindrical Wiener processes on a Hilbert space, but we mention that the generalization of our results to ordinary Wiener processes on Hilbert spaces is straight forward. For the definition and properties of stochastic integration with respect to cylindrical Wiener processes see Peszat and Zabczyk~\cite{Peszat-Zabczyk}.

Compensated Poisson random measures naturally arise when we have a L\'evy process. A L\'evy process is a process which has independent and stationary increments and can take values in a Hilbert or Banach space. General examples of L\'evy processes are Wiener process (Brownian motion), which is a L\'evy process with continuous trajectories, Poisson process, and compound Poisson processes.

For a detailed treatment of L\'evy processes and its relation to compensated Poisson random measures and definition and properties of integration with respect to compensated Poisson random measures we refer the reader to~\cite{Peszat-Zabczyk}.

\subsection{Monotone Operators}

Theory of monotone operators is an important branch in the theory of nonlinear equations and has first appeared in the works of Minty in 1961.

\begin{definition}
    $f:H\to H$ is called \emph{semi-monotone} if there exists a real constant $M$ such that
    \[ \forall x,y \in H: \qquad \langle f(x)-f(y) , x-y \rangle \le M \| x-y \|^2 \]
    and is called \emph{monotone} if $M=0$.
\end{definition}
{\noindent Note that semi-monotone condition is weaker than Lipschitz condition.}

Consider the following deterministic equation in a Banach space.
\begin{equation} \begin{array}{l}
    \frac{d}{dt}u(t)=f(t,u(t)),\qquad t\ge 0\\
    u(0)=u_0
\end{array}\end{equation}
The standard existence and uniqueness theorems for this equation, assume that $f$ is Lipschitz. In finite dimension, the continuity is sufficient to ensure the existence, but on Banach spaces there are examples which show that continuity is not sufficient (see for example,~\cite{Dieudonne}, problem 5, p.290). Works of Browder~\cite{Browder} and Kato~\cite{Kato} showed that on Hilbert spaces if we assume both continuity and semi-monotonicity then the solution exists and is unique. They also extended their proofs to semilinear evolution equations
\begin{equation} \begin{array}{l}
    \frac{d}{dt}u(t)=A(t) u(t) + f(t,u(t)),\qquad t\ge 0\\
    u(0)=u_0
\end{array}\end{equation}
where $A(t)$ is the generator of an evolution operator.

Kato~\cite{Kato} assumes \emph{demicontinuity} condition for $f$, which is weaker than continuity.

\begin{definition}
    $f:H\to H$ is called \emph{demicontinuous} if whenever $x_n \to x$, strongly in $H$ then $f(x_n)\rightharpoonup f(x)$ weakly in $H$.
\end{definition}

\begin{theorem}[Kato~\cite{Kato}]\label{Theorem: Kato}
    Let $f:[0,T]\times H \to H$ be semi-monotone, demicontinuous and maps bounded sets into bounded sets, Then the equation
    \[ \frac{du}{dt}=A u(t) + f(t,u(t)) \]
    with initial condition $u_0$ has a mild solution and the solution is unique.
\end{theorem}

For a proof of this theorem and other properties of monotone operators see~\cite{Tanabe}.

\subsection{The Main equation}

Let $H$ be a separable Hilbert space and $S(t)$ a $C_0$-semigroup of linear operators on $H$ with generator $A$. We are concerned with this equation,
\begin{equation}\label{main_equation}
    dX_t=AX_t dt+f(t,X_t) dt + g(t,X_{t-})d W_t + \int_E k(t,\xi,X_{t-}) \tilde{N}(dt,d\xi),
\end{equation}
where $W_t$ is a cylindrical Wiener process on another Hilbert space, $\tilde{N}(dt,d\xi)$ is a compensated Poisson random measure on a Banach space $U$ and independent of $W_t$. We assume $f$ is semimonotone and $g$ and $k$ are Lipschitz and have linear growth. In section~\ref{section: Assumptions} the assumptions on coefficients are stated precisely.

The main results of this article are Theorem~\ref{theorem: continuity I} and Corollary~\ref{corollary: continuity II}, proved in section~\ref{section: Continuity With Respect to Parameter} which state that the solutions equation~\eqref{main_equation} depend continuously, in an appropriate sense, on initial condition and also on coefficients. Several consequences of these results are also provided. In corollary~\ref{corollary: Exponential Stability} a sufficient condition for exponential asymptotic stability of the solutions is derived. In section~\ref{section: Yosida Approximation} we introduce the well known Yosida approximations of equation~\ref{main_equation} and show that the solutions of them converge to the solution of~\eqref{main_equation}. In section~\ref{section: Markov Property} the Markov property of the mild solutions is proved. We will provide some concrete examples to which our results apply. These examples consist of semilinear stochastic partial differential equations and a stochastic delay differential equation. Some of the statements have been presented previously in~\cite{Proceedings}.

\subsection{Stochastic Convolution Integrals} \label{subsection:stochastic_convolution_integrals}

Let $Z(t)$ be a stochastic process. Consider the equation $dX(t)= A X(t) dt + dZ(t)$ with an initial condition $X(0)$. Since $A$ is not defined on all of $H$ this equation may have no solutions, for example when $X(0)\notin Domain(A)$. In the case that $Z(t)$ is an $H$-valued semimartingale, a weaker notion of solution for this equation, i.e. \emph{mild solution} is defined as $X(t)=S(t) X(0)+\int_0^t S(t-s) dZ(s)$, where the integral is a stochastic integral. This is called a stochastic convolution integral (for the definition and properties of semimartingales and stochastic integration with respect to them the reader is referred to Metivier~\cite{Metivier}).

Now we introduce the concept of mild solution for~\eqref{main_equation}.
\begin{definition}
    By a \emph{mild solution} of equation~\eqref{main_equation} with initial condition $X_0$ we mean an adapted c\`adl\`ag process $X_t$ that satisfies
    \begin{multline}\label{mild_solution}
        X_t=S_t X_0+\int_0^t S_{t-s}f(s,X_s) ds+\int_0^t{S_{t-s}g(s,X_{s-})d W_s}\\
        + \int_0^t{\int_E {S_{t-s}k(s,\xi,X_{s-})} \tilde{N}(ds,d\xi).}
    \end{multline}
\end{definition}

Inequalities concerning upper bounds for the norm of stochastic convolution integrals are useful in studying stochastic evolution equations. One of the first such inequalities was that of Kotelenez~\cite{Kotelenez-1982} which is a maximal inequality for stochastic convolution integrals. Kotelenez~\cite{Kotelenez-1982} uses this inequality to prove the existence of a c\`adl\`ag version for stochastic convolution integrals. From now on, we always assume that stochastic convolution integrals are c\`adl\`ag. Later Kotelenez~\cite{Kotelenez-1984} proved a stronger inequality which was a stopped Doob inequality.

\begin{theorem}[Kotelenez,~\cite{Kotelenez-1984}] \label{Theorem: Kotelenez inequality}
    Assume $\alpha\ge 0$. There exists a constant $\mathbf{C}$ such that for any $H$-valued c\`adl\`ag locally square integrable martingale $M_t$ we have
    \[ \mathbb{E} \sup_{0\le t\le T} \|\int_0^t S_{t-s}dM_s\|^2 \le \mathbf{C} e^{4\alpha T} \mathbb{E}[M]_T.\]
\end{theorem}

\begin{remark}
    Hamedani and Zangeneh~\cite{Hamedani-Zangeneh-stopped} generalized this inequality to a stopped maximal inequality for $p$-th moment ($0<p<\infty$) of stochastic convolution integrals.
\end{remark}

Usual inequalities such as Theorem~\ref{Theorem: Kotelenez inequality} concern the expectation of the norm of stochastic convolution integrals and because of the presence of monotone nonlinearity in equation~\eqref{main_equation}, they are not applicable to~\eqref{main_equation}. For this reason we will use the following pathwise inequality for the norm of stochastic convolution integrals which has been proved in Zangeneh~\cite{Zangeneh-Paper}.
\begin{theorem}[It\^o type inequality, Zangeneh~\cite{Zangeneh-Paper}]\label{theorem:ito type inequality}
    Let $Z_t$ be an $H$-valued c\`adl\`ag locally square integrable semimartingale. If
    \[ X_t=S_t X_0 + \int_0^t S_{t-s}dZ_s, \]
    then, a.s.
    \begin{equation*}
        \lVert X_t \rVert ^2 \le e^{2\alpha t}\lVert X_0 \rVert ^2 + 2 \int_0^t {e^{2\alpha (t-s)}\langle X_{s-} , d Z_s \rangle}+\int_0^t {e^{2\alpha (t-s)}d[Z]_s},
    \end{equation*}
    where $[Z]_t$ is the quadratic variation process of $Z_t$.
\end{theorem}

\section{The Assumptions}\label{section: Assumptions}

Let $H$ be a separable Hilbert space with inner product $\langle \, , \, \rangle$. Let $S_t$ be a $C_0$ semigroup on $H$ with infinitesimal generator $A:D(A)\to H$. Furthermore we assume the exponential growth condition on $S_t$ holds, i.e. there exists a constant $\alpha$ such that $\| S_t \| \le e^{\alpha t}$. If $\alpha=0$, $S_t$ is called a contraction semigroup. We denote by $L_{HS}(K,H)$ the space of Hilbert-Schmidt mappings from a Hilbert space $K$ to $H$.

Let $(\Omega,\mathcal{F},\mathcal{F}_t,\mathbb{P})$ be a filtered probability space. Let $(E,\mathcal{E})$ be a measurable space and $N(dt,d\xi)$ a Poisson random measure on $\mathbb{R}^+ \times E$ with intensity measure $dt \nu(d\xi)$. Our goal is to study equation~\eqref{main_equation} in $H$, where $W_t$ is a cylindrical Wiener process on a Hilbert space $K$ and $\tilde{N}(dt,d\xi)=N(dt,d\xi)-dt\nu(d\xi)$ is the compensated Poisson random measure corresponding to $N$. We assume that $N$ and $W_t$ are independent. We also assume the following,

\begin{hypothesis}\label{main_hypothesis}
    \begin{description}

        \item[(a)] $f(t,x,\omega):\mathbb{R}^+\times H\times \Omega \to H$ is measurable, $\mathcal{F}_t$-adapted, demicontinuous with respect to $x$ and there exists a constant $M$ such that
            \[ \langle f(t,x,\omega)-f(t,y,\omega),x-y \rangle \le M \|x-y\|^2,\]

        \item[(b)] $g(t,x,\omega):\mathbb{R}^+\times H\times \Omega \to L_{HS}(K,H)$ and $k(t,\xi,x,\omega):\mathbb{R}^+\times E\times H\times \Omega \to H$ are predictable and there exists a constant $C$ such that
            \[ \| g(t,x,\omega)-g(t,y,\omega)\|_{L_{HS}(K,H)}^2 + \int_{E}\|k(t,\xi,x)-k(t,\xi,y)\|^2 \nu(d\xi) \le C \|x-y \|^2,\]

        \item[(c)] There exists a constant $D$ such that
            \[ \| f(t,x,\omega)\|^2 + \| g(t,x,\omega)\|_{L_{HS}(K,H)}^2 + \int_{E}\|k(t,\xi,x)\|^2 \nu(d\xi) \le D(1+\|x\|^2),\]
        \item[(d)] $X_0(\omega)$ is $\mathcal{F}_0$ measurable and square integrable.
    \end{description}

\end{hypothesis}

The following theorem states that equation~\eqref{main_equation} has a unique mild solution. For the proof see \cite{Archive-existence}.

\begin{theorem}[Existence and Uniqueness of the Mild Solution]\label{theorem:existence and uniqueness}
    Under the assumptions of Hypothesis~\ref{main_hypothesis}, equation~\eqref{main_equation} has a unique square integrable c\`adl\`ag mild solution with initial condition $X_0$.
\end{theorem}

\section{The Main Result}\label{section: Continuity With Respect to Parameter}

\begin{theorem} [Continuity With Respect to Parameter I]\label{theorem: continuity I}
    Assume that for $n=0,1$, $f_n(t,x,\omega),g_n(t,x,\omega)$ and $k_n(t,\xi,x,\omega)$ satisfy Hypothesis~\ref{main_hypothesis} with the same constants. Let $X^n_t$ be the unique mild solution of
    \begin{multline*}
        dX^n_t= A X^n_t dt+f_n(t,X^n_t) dt + g_n(t,X^n_{t-})d W_t
        + \int_E k_n(t,\xi,X^n_{t-}) \tilde{N}(dt,d\xi),
    \end{multline*}
    with initial condition $X^n_0$. Then,
    \begin{multline}\label{equation: continuity}
        \mathbb{E} \sup\limits_{0\le t\le T} e^{-2\alpha t} \| X^1_t-X^0_t \| ^2 \le 2 e^{C_1T} \mathbb{E}\|X^1_0-X^0_0\|^2 \\
        + 2 e^{C_1 T} \int_0^T e^{-2\alpha t} \mathbb{E}\|f_1(t,X^0_t)-f_0(t,X^0_t)\|^2 dt\\
        + C_2 e^{C_1 T}\int_0^T e^{-2\alpha t} {\mathbb{E}\|(g_1(t,X^0_t)-g_0(t,X^0_t))\|^2 dt}\\
        + C_2 e^{C_1 T}\int_0^T \int_E{ e^{-2\alpha t} \mathbb{E}\|(k_1(t,\xi,X^0_t)-k_0(t,\xi,X^0_t))\|^2 \nu(d\xi) dt},
    \end{multline}
    for $C_1=4 M + 2 + C (8\mathcal{C}_1^2+4)$ and $C_2=8\mathcal{C}_1^2+4$ where $\mathcal{C}_1$ is the constant in Burkholder-Davies-Gundy inequality.
\end{theorem}
\begin{proof}
    First we consider the case that $\alpha=0$. Subtract $X^1$ and $X^0$,
    \begin{multline*}
        X^1_t-X^0_t=S_t (X^1_0-X^0_0)\\
        + \int_0^t S_{t-s} (f_1(s,X^1_s)-f_0(s,X^0_s))ds
        + \int_0^t S_{t-s} dM_s,
    \end{multline*}
    where
    \[ M_t=\int_0^t (g_1(s,X^1_{s-})-g_0(s,X^0_{s-}))dW_s+\int_E (k_1(s,\xi,X^1_{s-})-k_0(s,\xi,X^0_{s-}))d\tilde{N}. \]
    Applying It\^o type inequality (Theorem~\ref{theorem:ito type inequality}), for $\alpha=0$, to $X^1-X^0$ we find
    \begin{multline}\label{equation:proof of continuity 1}
        \| X^1_t-X^0_t \| ^2 \le \| X^1_0-X^0_0 \| ^2 +  2 \underbrace {\int_0^t {\langle X^1_{s-}-X^0_{s-} , (f_1(s,X^1_s)-f_0(s,X^0_s))\rangle ds}}_\bold{A_t}\\
        + 2 \underbrace {\int_0^t {\langle X^1_{s-}-X^0_{s-} , d M_s \rangle}}_\bold{B_t}+ [M]_t.
    \end{multline}
    We have
    \begin{multline*}
        \bold{A_t} =  \int_0^t {\langle X^1_{s-} - X^0_{s-} , f_1(s,X^1_s)-f_1(s,X^0_s) \rangle ds}\\
        + \int_0^t \langle X^1_{s-}-X^0_{s-} , f_1(s,X^0_s)-f_0(s,X^0_s) \rangle ds.
    \end{multline*}
    Using the monotonicity assumption and Cauchy-Schwartz inequality we have
    \begin{multline}\label{equation:proof of continuity 3}
        \bold{A_t} \le M \int_0^t \|X^1_{s}-X^0_{s}\|^2 ds + \frac{1}{2}\int_0^t \|X^1_{s}-X^0_{s}\|^2 ds\\
        + \frac{1}{2} \int_0^t \|f_1(s,X^0_{s})-f_0(s,X^0_{s})\|^2 ds.
    \end{multline}
    Applying Burkholder-Davies-Gundy inequality for $p=1$ to term $\mathbf{B_t}$ we find
        \[ \mathbb{E}\sup\limits_{0\le s\le t}\mathbf{|B_s|} \le \mathcal{C}_1\mathbb{E}\left(\sup\limits_{0\le s\le t} \|X^1_s-X^0_s\|[M]_t^\frac{1}{2}\right),\]
    and by Cauchy-Schwartz inequality,
    \begin{equation}\label{equation:proof of continuity 4}
        \le \frac{1}{4} \mathbb{E}\sup\limits_{0\le s\le t} \|X^1_s-X^0_s\|^2 + \mathcal{C}_1^2 \mathbb{E}[M]_t.
    \end{equation}
    We have
    \begin{eqnarray*}
        \mathbb{E}[M]_t &=& \int_0^t{\mathbb{E}\|(g_1(s,X^1_s)-g_0(s,X^0_s))\|^2 ds}\\
        &&+ \int_0^t\int_E{\mathbb{E}\|(k_1(s,\xi,X^1_s)-k_0(s,\xi,X^0_s))\|^2 \nu(d\xi) ds}\\
        & \le & 2 \int_0^t{\mathbb{E}\|(g_1(s,X^1_s)-g_1(s,X^0_s))\|^2 ds}\\
        && + 2\int_0^t{\mathbb{E}\|(g_1(s,X^0_s)-g_0(s,X^0_s))\|^2 ds} \\
        && + 2\int_0^t\int_E{\mathbb{E}\|(k_1(s,\xi,X^1_s)-k_1(s,\xi,X^0_s))\|^2 \nu(d\xi) ds}\\
        && +2\int_0^t\int_E{\mathbb{E}\|(k_1(s,\xi,X^0_s)-k_0(s,\xi,X^0_s))\|^2 \nu(d\xi) ds}.
    \end{eqnarray*}
    Using the Lipschitz assumption on $g$ and $k$ we find
    \begin{multline}\label{equation:proof of continuity 5}
        \mathbb{E}[M]_t \le 2C \int_0^t \mathbb{E}\|X^1_s-X^0_s\|^2 ds\\
        + 2\int_0^t{\mathbb{E} \|(g_1(s,X^0_s)-g_0(s,X^0_s))\|^2 ds}\\
        + 2\int_0^t\int_E{\mathbb{E} \|(k_1(s,\xi,X^0_s)-k_0(s,\xi,X^0_s))\|^2 \nu(d\xi) ds}.
    \end{multline}
    Substituting \eqref{equation:proof of continuity 3}, \eqref{equation:proof of continuity 4} and \eqref{equation:proof of continuity 5} in \eqref{equation:proof of continuity 1}, after cancellation we find
    \begin{eqnarray*}
        \mathbb{E} \sup\limits_{0\le s\le t} \| X^1_s-X^0_s \| ^2 &\le&  C_1 \int_0^t \mathbb{E}\|X^1_s-X^0_s\|^2 ds + 2 \mathbb{E}\|X^1_0-X^0_0\|^2  \\
        && + 2 \int_0^t \mathbb{E}\|f_1(s,X^0_s)-f_0(s,X^0_s)\|^2 ds\\
        && + C_2\int_0^t{\mathbb{E}\|(g_1(s,X^0_s)-g_0(s,X^0_s))\|^2 ds}\\
        && + C_2\int_0^t\int_E{\mathbb{E}\|(k_1(s,\xi,X^0_s)-k_0(s,\xi,X^0_s))\|^2 \nu(d\xi) ds},
    \end{eqnarray*}
    where $C_1=4 M + 2 + C(8\mathcal{C}_1^2+4)$ and $C_2=8\mathcal{C}_1^2+4$.

    Now applying Gronwall's inequality the statement follows. 	Hence the proof for the case $\alpha=0$ is complete. Now for the general case, we apply the following change of variables,
	        \begin{gather*}
           \tilde{S}_t= e^{-\alpha t} S_t ,\qquad \tilde{f}(t,x,\omega)=e^{-\alpha t}f(t,e^{\alpha t}x,\omega) ,\qquad \tilde{g}(t,x,\omega)=e^{-\alpha t}g(t,e^{\alpha t}x,\omega), \\
           \tilde{k}(t,\xi,x,\omega)=e^{-\alpha t}k(t,\xi,e^{\alpha t}x,\omega).
        \end{gather*}
        Note that $\tilde{S}_t$ is a contraction semigroup. It is easy to see that $X_t$ is a mild solution of equation~\eqref{main_equation} if and only if $\tilde{X}_t=e^{-\alpha t} X_t$ is a mild solution of equation with coefficients $\tilde{S},\tilde{f},\tilde{g},\tilde{k}$.
        
\end{proof}

As a consequence of Theorem~\ref{theorem: continuity I} we prove that if the coefficients and initial conditions of a sequence of equations converge, then their mild solutions also converge to the mild solution of the limiting equation. The convergence that we prove is in a stronger sense than similar result in~\cite{Albeverio-Mandrekar-Rudiger-2009}.

\begin{corollary}[Continuity With Respect to Parameter II]\label{corollary: continuity II}
    Assume that for $n=0,1,2,\ldots$, $f_n$, $g_n$, $k_n$ and $X^n_0$ satisfy Hypothesis~\ref{main_hypothesis} with same constants and assume that for every $t \in [0,T]$ and $x\in H$ we have almost surely
    \begin{eqnarray*}
        &f_n(t,x,\omega)\to f_0(t,x,\omega)&\\
        &g_n(t,x,\omega)\to g_0(t,x,\omega)&\\
        &\int_E{\|k_n(t,\xi,x,\omega)-k_0(t,\xi,x,\omega)\|^2 \nu(d\xi)} \to 0&\\
        &\mathbb{E}\|X^n_0-X^0_0\|^2 \to 0.&
    \end{eqnarray*}
    Then
    \begin{equation*}
        \mathbb{E} \sup\limits_{0\le t\le T} \|X^n_t-X^0_t\|^2  \to 0.
    \end{equation*}
\end{corollary}

\begin{proof}
    Apply Theorem~\ref{theorem: continuity I} to $X^n$ and $X^0$. Note that by Hypothesis~\ref{main_hypothesis}-(c) the integrands on the right hand side of~\eqref{equation: continuity} are dominated by a constant multiple of $(1+\|X^0_t(\omega)\|^2)$, on the other hand by assumptions they tend to zero almost everywhere on $[0,T]\times\Omega$. Hence by dominated convergence theorem, the right hand side of~\eqref{equation: continuity} tends to $0$ and therefore
        \[ \mathbb{E} \sup\limits_{0\le t\le T} e^{-2\alpha t} \| X^1_t-X^0_t \| ^2 \to 0 \]
    which implies the statement.
\end{proof}

As another consequence of Theorem~\ref{theorem: continuity I} it follows that if the contraction coefficient of the semigroup is sufficiently negative, then all the mild solutions are exponentially stable.

\begin{corollary} [Exponential Stability]\label{corollary: Exponential Stability}
    Let $X_t$ and $Y_t$ be mild solutions of~\eqref{main_equation} with initial conditions $X_0$ and $Y_0$. Then
    \begin{eqnarray*}
        \mathbb{E} \| X_t-Y_t \| ^2 &\le& 2 e^{\gamma t} \mathbb{E}\|X_0-Y_0\|^2
    \end{eqnarray*}
    for $\gamma= 2\alpha + 4 M + 2 + C(8\mathcal{C}_1^2+4)$. In particular, if $\gamma < 0$ then all mild solutions are exponentially stable.
\end{corollary}

\section{Yosida Approximations}\label{section: Yosida Approximation}

As another application of Theorem~\ref{theorem: continuity I} we construct Lipschitz approximations of equation~\eqref{main_equation} known as Yosida approximations and prove that their solutions converge to the solution of~\eqref{main_equation}. Since equations with Lipschitz coefficients can be solved numerically, this can be used as a scheme for numerical solution of~\eqref{main_equation}.

In this section we assume $f:H\to H$ satisfies a condition stronger than monotonicity which is maximal monotonicity. This concept in its most generality is defined for subsets of $X\times X^*$ which $X$ is a Banach space. For a detailed treatment of this concept see~\cite{Barbu}.

\begin{definition}
$A\subset X\times X^*$ is called monotone, if for any $(x_1,y_1),(x_2,y_2)\in A$,
\[ (y_2-y_1,x_2-x_1) \le 0 \]
and is called \emph{maximal} monotone if it is monotone and is not properly contained in any monotone set.
\end{definition}

Note that any operator $f:X\to X^*$ can be viewed as a subset of $X\times X^*$ and hence the concept of maximal monotonicity is defined for it, especially since $H$ is a Hilbert space, maximal monotonicity makes sense for operators $f:H\to H$.

Maximal monotonicity is not very restrictive as the following theorem shows that any monotone operator with a weak continuity assumption called \emph{hemicontinuity} is maximal monotone.

\begin{definition}
$f:H\to H$ is called hemicontinuous if for any $x,y\in H$, $f(x+ty)$ is continuous as a function of $t\in\mathbb{R}$ (in other words $f$ is continuous in each direction).
\end{definition}

\begin{theorem}[\cite{Barbu}, page 45, Theorem 1.3]\label{theorem: hemicontinuous}
Let $f:H\to H$ be monotone and hemicontinuous, then it is maximal monotone.
\end{theorem}

Let $f:H\to H$ be a maximal monotone operator. Then we follow~\cite{Barbu} to define for $\lambda>0$,
\[ I_\lambda = (I-\lambda f)^{-1} \]
\[ f_\lambda = \lambda^{-1} (I_\lambda-I) \]
Note that with our notation, $-f$ is maximal monotone in the sense of~\cite{Barbu}. Some of the important properties of $f_\lambda$ are listed in the following proposition.

\begin{proposition}[\cite{Barbu}, page 49, Proposition 1.3]\label{proposition: Yosida Properties}
Let $f:H\to H$ be maximal monotone. Then
\begin{description}
	\item[(i)] $f_\lambda$ is monotone and Lipschitz on $H$.
	\item[(ii)] For any $x\in H$, $\|f_\lambda (x)\| \le \|f(x)\|$.
	\item[(iii)] For any $x\in H$, $\lim_{\lambda\to 0} f_\lambda (x) = f(x)$ strongly in $H$.
\end{description}
\end{proposition}

Now we are ready to state and prove the main theorem of this section.
\begin{theorem}\label{theorem: Yosida approximation}
	Let $f$ be maximal monotone and let $X^\lambda$ be the mild solution of
	\begin{equation}\label{approximate_equation}
    dX^\lambda_t=AX^\lambda_t dt+f_\lambda(X^\lambda_t) dt + g(t,X^\lambda_{t-})d W_t + \int_E k(t,\xi,X^\lambda_{t-}) \tilde{N}(dt,d\xi),
	\end{equation}
	then we have
	\[ \lim_{\lambda\to 0} \mathbb{E} \left(\sup_{0\le s\le t} \|X^\lambda_s - X_s\|^2 \right) = 0 \]
\end{theorem}

\begin{proof}
	By Proposition~\ref{proposition: Yosida Properties}-(i), $f_\lambda$'s are monotone and continuous and by (ii) they have linear growth condition with the same constant as $f$. Hence the assumptions of Corollary~\ref{corollary: continuity II} are satisfied and the statement follows.
\end{proof}

\begin{remark}
	The hemicontinuity assumption holds for many monotone operators, especially for Nemitsky operators associated with decreasing continuous real functions, since as will be mentioned in section~\ref{section:examples} these operators are in fact continuous and hence hemicontinuous on $L^2(\mathcal{D})$ and therefore they are maximal monotone by Theorem~\ref{theorem: hemicontinuous}. Hence Theorem~\ref{theorem: Yosida approximation} could be applied to examples of section~\ref{section:examples}.
\end{remark}

\section{Markov Property}\label{section: Markov Property}

In this section we assume that $f$, $g$ and $k$ are deterministic functions and satisfy Hypothesis~\ref{main_hypothesis}. Let $0\le s \le t$ and $\eta: \Omega \to H$ be $\mathcal{F}_s$-measurable and square integrable. We denote by $X(s,\eta,t)$ the value at time $t$ of the solution of~\eqref{main_equation} starting at time $s$ from $\eta$. Let $B_b(H)$ be the space of real valued bounded measurable functions on $H$. For $\varphi \in B_b(H)$ and $x\in H$ define
\[ P_{s,t}\varphi (x):= \mathbb{E}\varphi(X(s,x,t)). \]
$P_{s,t}$ is called the \emph{transition semigroup}.

\begin{theorem}[Markov Property]\label{theorem: markov property}
    For $0\le r \le s \le t$ and $\varphi\in B_b(H)$ we have almost surely
    \[ \mathbb{E}\left( \varphi(X(r,x,t)|\mathcal{F}_s \right) = P_{s,t}\varphi (X(r,x,s)) \qquad \mathbb{P}-\mathrm{almost\,\,sure}. \]
\end{theorem}

\begin{proof}
    Let $C_b(H)$ denote the set of real valued bounded continuous functions on $H$. It suffices to prove the theorem for $\varphi\in C_b(H)$ since every $\varphi\in B_b(H)$ is the pointwise limit of a uniformly bounded sequence in $C_b(H)$. Fix $r$, $s$ and $t$. We claim that for any square integrable random variable $\eta(\omega)$ which is $\mathcal{F}_s$ measurable, we have
    \begin{equation}\label{equation: proof of Markov 1}
        \mathbb{E}\left( \varphi(X(s,\eta,t))|\mathcal{F}_s\right)=P_{s,t}\varphi(\eta(\omega)) \qquad \mathbb{P}-\mathrm{almost\,\,sure}.
    \end{equation}
    We first prove the claim for the case that $\eta$ has a simple form $\eta=\sum y_k \chi_{A_k}$, where $y_k\in H$ and $A_k\in\mathcal{F}_s$ form a partition of $\Omega$. We have
    \begin{eqnarray*}
        \mathbb{E}\left( \varphi(X(s,\eta,t))|\mathcal{F}_s\right)&=&\mathbb{E}\left( \sum \varphi(X(s,y_k,t)) \chi_{A_k}\big|\mathcal{F}_s\right)\\
        &=& \sum \chi_{A_k} \mathbb{E}\left( \varphi(X(s,y_k,t))|\mathcal{F}_s\right).
    \end{eqnarray*}
    Note that $X(s,y_k,t)$ is independent of $\mathcal{F}_s$, hence
    \begin{eqnarray*}
        &=& \sum \chi_{A_k} \mathbb{E}\left(\varphi(X(s,y_k,t))\right)\\
        &=& \sum \chi_{A_k} P_{s,t}\varphi(y_k)=P_{s,t}\varphi(\eta(\omega)).
    \end{eqnarray*}
    Now for general $\eta$ choose a sequence $\eta_n$ of simple random variables such that tend to $\eta$ in $L^2(\Omega)$ and almost surely. We then have
    \[ \mathbb{E}\left( \varphi(X(s,\eta_n,t))|\mathcal{F}_s\right)=P_{s,t}\varphi(\eta_n(\omega)) \qquad \mathbb{P}-\mathrm{almost\,\,sure}. \]
    Now let $n\to \infty$. By continuity with respect to initial conditions, the left hand side converges to $\mathbb{E}\left( \varphi(X(s,\eta,t))|\mathcal{F}_s\right)$ and the right hand side converges to $P_{s,t}\varphi(\eta(\omega))$ and~\eqref{equation: proof of Markov 1} follows. Now in~\eqref{equation: proof of Markov 1} let $\eta(\omega)=X(r,x,s)$. By uniqueness of solution we have $X(r,x,t)=X(s,X(r,x,s),t)$ and the theorem follows.
\end{proof}

\section{Some Examples}\label{section:examples}

In this section we provide some concrete examples of semilinear stochastic evolution equations with monotone nonlinearity and L\'evy noise which the results of previous sections could be applied. The examples consist of stochastic partial differential equations of parabolic and hyperbolic type and a stochastic delay differential equation.

\begin{example}[Stochastic reaction-diffusion equations with multiplicative Poisson noise]

\label{example: finite_diemnsional_noise}
    In this example we consider a class of semilinear stochastic evolution equations with multiplicative Poisson noise. Let $\mathcal{D}$ be a bounded domain with a smooth boundary in $\mathbb{R}^d$. Consider the equation,

    \begin{equation}\label{equation: example_finite_dimensional_noise}
        \left\{\begin{array}{lll}
            d u (t)  &=& Au(t) dt + f(u(t,x)) dt + \eta u(t) dt + \int_E k(t,\xi,u(t^-,x)) \tilde{N}(dt,d\xi)\\
            u(0) & = & u_0.
        \end{array} \right.
    \end{equation}
	where $A$ is the generator of a $C_0$ semigroup on $L^2(\mathcal{D})$, $f:\mathbb{R}\to\mathbb{R}$ is a continuous decreasing function with linear growth and $k:[0,T]\times E \times \mathbb{R}\times\Omega\to \mathbb{R}$ is measurable and satisfies the Lipschitz condition
	\[ \mathbb{E} \int_E |k(s,\xi,u)-k(s,\xi,v)|^2 \mu(d\xi) \le C |u-v|^2 \]
and the linear growth condition
	\[ \mathbb{E} \int_E |k(s,\xi,u)|^2 \mu(d\xi) \le D (1+|u|^2) \]
and $u_0\in L^2(\mathcal{D})$. We prove that equation~\eqref{equation: example_finite_dimensional_noise} has a unique mild solution in $L^2(\mathcal{D})$.

	We show that equation~\eqref{equation: example_finite_dimensional_noise} satisfies the assumptions of Theorem~\ref{theorem:existence and uniqueness}. Let $H=L^2(\mathcal{D})$. We denote the Nemitsky operator associated with a function $f:\mathbb{R}\to\mathbb{R}$ by the same symbol. Since $f$ and $k$ are continuous and have linear growth, by Theorem (2.1) of Krasnosel'ski\u\i~\cite{Krasnoelskii}, the associated Nemitsky operators define continuous operators from $L^2(\mathcal{D})$ to $L^2(\mathcal{D})$ and have linear growth. Verifying the other assumptions is straight forward. Hence applying Theorem~\ref{theorem:existence and uniqueness} we conclude that equation~\eqref{equation: example_finite_dimensional_noise} has a unique mild solution in $L^2(\mathcal{D})$ and Theorem~\ref{theorem: continuity I} implies that the solution map $u_0 \mapsto u$ is Lipschitz in the sense that
	\[ \mathbb{E}\sup_{t\le T} \|u(t)-v(t)\|_{L^2(\mathcal{D})}^2 \le C \|u_0-v_0\|_{L^2(\mathcal{D})}^2 \]
We also can use Thoerem~\ref{theorem: Yosida approximation} to build Lipschitz approximations of equation~\eqref{equation: example_finite_dimensional_noise} and solve them numerically to approximate the solution of ~\eqref{equation: example_finite_dimensional_noise}.

\begin{remark}
	\begin{enumerate}
		\item Equation~\eqref{equation: example_finite_dimensional_noise} is exactly the same as the main equation studied in~\cite{Marinelli-Rockner-wellposedness}.
		\item As important examples for the operator $A$, one can denote any second order elliptic operator on $\mathcal{D}$.
		\item The same results hold if we add a Wiener noise term with Lipschitz coefficient.
		\item It is straight forward to generalise this example to the case that $f$ and $k$ depend also on $x$ and in that case it suffices to assume that $f(x,u)$ and $k(x,u)$ satisfy caratheodory condition, i.e they are continuous with respect to $u$ for almost all $x\in \mathcal{D}$ and are measurable with respect to $x$ for all values of $u$.
	\end{enumerate}
\end{remark}

\end{example}

\begin{example}[Second Order Stochastic Hyperbolic Equations with L\'evy noise] \label{example: finite_diemnsional_noise_hyperbolic}

	In this example we consider a hyperbolic SPDE with L\'evy noise. Let $\mathcal{D}$ be a bounded domain with a smooth boundary in $\mathbb{R}^d$, Consider the initial boundary value problem,
    \begin{equation}\label{equation: example_finite_dimensional_noise_hyperbolic}
        \left\{\begin{array}{lrll}
            \frac{\partial^2 u}{\partial t^2}  = \Delta u -\sqrt[3]{\frac{\partial u}{\partial t}} & + u(t^-,x) \frac{\partial Z}{\partial t}& \textrm{on} & [0,\infty) \times \mathcal{D}\\
            u =0 && \textrm{on} & [0,\infty) \times \partial \mathcal{D}\\
            u(0,x) = u_0(x) && \textrm{on} & \mathcal{D}.\\
            \frac{\partial u}{\partial t} (0,x) = 0 && \textrm{on} & \mathcal{D}.
        \end{array} \right.
    \end{equation}
    where $Z(t)$ is a real valued square integrable L\'evy process and $u_0(x) \in L^2(\mathcal{D})$ is the initial condition.
    We apply the results of previous sections and conclude that this equation has a unique mild solution in $H^1(\mathcal{D})$ (Sobolev space of weakly differentiable functions on $\mathcal{D}$ with derivative in $L^2(\mathcal{D})$). One can replace $-\sqrt[3]{x}$ by any continuous decreasing real function with linear growth. We generalize this equation as follows:

    \begin{equation}\label{equation: example_hyperbolic_finite_noise}
        \left\{\begin{array}{llrll}
             \frac{\partial^2}{\partial t^2} u(t,x)& = & \Delta u + f(u,\frac{\partial u}{\partial t}) + g_i(u(t^-,x)) \frac{\partial W_i}{\partial t} & & \\
             && + k_j(u(t^-,x)) \frac{\partial Z_j}{\partial t}  & \textrm{on} &  [0,\infty) \times \mathcal{D} \\
             u&=0 && \textrm{on} &  [0,\infty) \times \partial \mathcal{D} \\
             u(0,x)&=u_0(x)& & \textrm{on} & \mathcal{D} \\
            \frac{\partial u}{\partial t} (0,x)&=0 && \textrm{on} & \mathcal{D}.
          \end{array} \right.
    \end{equation}
	where $W_i(t), i=1,\ldots ,m$ are standard Wiener processes in $\mathbb{R}$ and $Z_j(t), j=1,\ldots ,n$ are pure jump L\'evy martingales in $\mathbb{R}$ with intensity measures $\nu_j(d\xi)$ and $u_0(x)\in L^2(\mathcal{D})$.

	Moreover assume that,
    \begin{hypothesis}\label{hypothesis: Second Order Hyperbolic finite dimensional noise}
        \begin{description}

            \item[(a)] $f:\mathbb{R}\times\mathbb{R}\to \mathbb{R}$ is measurable and continuous and is Lipschitz w.r.t first variable and semimonotone w.r.t second variable, i.e there exist constants $M$ and $C$ such that for any $a,a_1,a_2,b,b_1,b_2 \in \mathbb{R}$,
                \[ \|f(a_1,b)-f(a_2,b)\|\le C \|a_1-a_2\|. \]
                \[  f(a,b_1)-f(a,b_2) \le M (b_1-b_2), \]

            \item[(b)] There exists a constant $C>0$ such that for any $x\in \mathcal{D}$ and $a,b\in\mathbb{R}$,
                \[ \sum\limits_{i=1}^m |g_i(x,a)-g_i(x,b)|^2 + \sum\limits_{j=1}^n |k_j(x,a)-k_j(x,b)|^2 \le C |a-b|^2.\]

            \item[(c)] There exists a constant $D>0$ such that for any $a,b \in\mathbb{R}$,
                \[ |f(a,b)| + \sum\limits_{i=1}^m |g_i(a)| + \sum\limits_{j=1}^n |k_j(a)| \le D (1 + |a|+|b|).\]

        \end{description}
    \end{hypothesis}

	Note that $\Delta$ is self adjoint and negative definite on $L^2$. Moreover, we have
        \[ D((-\Delta)^\frac{1}{2})=  H^1(\mathcal{D}). \]
    Hence by Lemma B.3 of~\cite{Peszat-Zabczyk}, the operator
        \[ \mathcal{A}=\left( {\begin{array}{cc} 0&I\\  \Delta &0 \end{array}} \right)\]
    generates a $C_0$ semigroup of contractions on $H$.

    Let $K=E=\mathbb{R}$. We also define for $(u,v) \in H$ and $\phi\in K$ and $\xi\in E$,
        \[ \bar{f}(u,v)= \left( {\begin{array}{c} 0\\ f(u(x),v(x)) \end{array}} \right), \bar{g}(u,v)(\phi)= \left( {\begin{array}{c} 0\\ g(u(x)) \phi \end{array}} \right), \bar{k}(\xi,u,v)= \left( {\begin{array}{c} 0\\ k(u(x)) \xi \end{array}} \right) \]
We claim that $\bar{f}$, $\bar{g}$ and $\bar{k}$ satisfy Hypothesis~\ref{main_hypothesis}. The continuity of $\bar{f}$, $\bar{g}$ and $\bar{k}$ follows as in example~\ref{example: finite_diemnsional_noise} (note that the values of these functions are essentially in $L^2(\mathcal{D})$ and that $H^1(\mathcal{D})$ embeds continuously in $L^2(\mathcal{D})$). We show the semimonotonicity condition, the other conditions are straightforward.
        \begin{eqnarray*}
            \langle \bar{f}(u_1,v_1) - \bar{f}(u_2,v_2) , \left( {\begin{array}{c} u_1\\ v_1 \end{array}} \right) - \left( {\begin{array}{c} u_2\\ v_2 \end{array}} \right) \rangle = \langle f(u_1,v_1) - f(u_2,v_2) , v_1-v_2 \rangle &&\\
            = \langle f(u_1,v_1) - f(u_1,v_2) , v_1-v_2 \rangle + \langle f(u_1,v_2) - f(u_2,v_2) , v_1-v_2 \rangle &&
        \end{eqnarray*}
    where by Hypothesis~\ref{hypothesis: Second Order Hyperbolic}-(a) and Shwartz inequality
        \begin{eqnarray*}
            \le M \|v_1-v_2\|^2 + C \|u_1-u_2\| \|v_1-v_2\| \le (M+C)\left( \|u_1-u_2\|^2 + \|v_1-v_2\|^2 \right)
        \end{eqnarray*}
    Hence Hypothesis~\ref{main_hypothesis}-(a) holds with constant $M+C$. Now, if we let
        \[X(t)= \left( {\begin{array}{c} u(t)\\ \frac{\partial u}{\partial t}(t) \end{array}} \right) \]
    then equation~\eqref{equation: example_hyperbolic} can be written as
        \[dX(t)=\mathcal{A}X(t) dt+ \bar{f}(X(t))dt+ \bar{g}(X(t^-))dW_t + \int_E \bar{k}(\xi,X(t^-)) \tilde{N}(dt,d\xi) \]
    and hence by Theorem~\ref{theorem:existence and uniqueness} has a mild solution $u(t,x,\omega)$ with values in $H$ and with c\`adl\`ag trajectories.

\end{example}

\begin{example}[SPDE with Space-Time Noise] \label{example:general_parabolic}
    In this example we would like to consider a SPDE with infinite dimensional noise. A natural candidate for infinite dimensional noise is space-time white noise, but it can be shown that in dimensions greater than one, even the equation
        \[ \frac{\partial u}{\partial t}(t,x) = \Delta u(t,x) + \dot{W} (t,x) \]
    does not have a function valued solution (\cite{Peszat-Zabczyk}, Remark 12.2). In order to guarantee the existence of solution we assume that coefficients are operators on certain function spaces.

    Let $\mathcal{D}$ be as in Example~\ref{example: finite_diemnsional_noise}. Consider the initial boundary value problem on $\mathcal{D}$,
    \begin{equation}\label{equation: example_parabolic}
        \left\{\begin{array}{lrll}
             \frac{\partial u}{\partial t} =& \Delta u + f(u(t)) + g(u(t^-)) \frac{\partial W}{\partial t} & & \\
             &+ k(u(t^-)) \frac{\partial Z}{\partial t} & \textrm{on} & [0,\infty) \times \mathcal{D}  \\
             u=0 && \textrm{on} &  [0,\infty) \times \partial \mathcal{D} \\
             u(0,x)=0 && \textrm{on} & \mathcal{D}
          \end{array} \right.
    \end{equation}
    where $W_t$ is a cylindrical Wiener process on $L^2(\mathcal{D})$ and $Z_t$ is a pure jump L\'evy martingale on $L^2(\mathcal{D})$, and by $u(t)$ we mean $u(t,.)$.

    Let $n$ be an integer. We wish to solve this equation in the function space $H_n$ introduced in Walsh~\cite{Walsh}. Let $\{\phi_j\}$ be the complete orthonormal basis for $L^2(\mathcal{D})$ consisting of eigenfunctions of $\Delta$ with Dirichlet boundary condition and $-\lambda_j<0$ be the corresponding eigenvalues. Let $H_n$ be the Hilbert space that has as a complete orthonormal basis the set $\{e_j=(1+\lambda_j)^{-\frac{n}{2}} \phi_j\}$. Obviously $H_0=L^2(\mathcal{D})$ and the spaces $H_n$ can be continuously embedded in each other as
        \[ \cdots \subset H_n \subset \cdots \subset H_1 \subset L^2(\mathcal{D}) \subset H_{-1} \subset \cdots \subset H_{-n} \subset \cdots. \]

    Assume moreover,

    \begin{hypothesis}\label{hypothesis: space-time noise parabolic}
        \begin{description}

            \item[(a)] $f:H_n \to H_n$ is measurable, demicontinuous and there exists a constant $M$ such that for any $u,v \in H_n$,
                \[ \langle f(u)-f(v),u-v \rangle \le M \|u-v\|^2,\]

            \item[(b)] $g:H_n \to L_{HS}(L^2(\mathcal{D}),H_n)$ and $k:H_n \to L(L^2(\mathcal{D}),H_n)$ are Lipschitz.

            \item[(c)] There exists a constant $D$ such that for $u\in H_n$,
                \[ \| f(u)\|^2 + \| g(u)\|^2 + \|k(u)\|^2 \le D (1+\|u\|^2),\]

        \end{description}
    \end{hypothesis}

    $A$ generates a $C_0$ semigroup $S_t$ on $H$ where $S_t e_j = e^{-t\lambda_j} e_j$. Let $K=E=L^2(\mathcal{D})$ and let $\tilde{N}(dt,d\xi)$ be the compensated Poisson random measure on $E$ corresponding to the L\'evy process $Z_t$ with intensity measure $\nu(d\xi)$, and define
        \[ \bar{k}(\xi,u):= k(u)(\xi) \]
    Now, it is easy to verify that $f$, $g$ and $\bar{k}$ satisfy Hypothesis~\ref{main_hypothesis} and therefore equation~\eqref{equation: example_parabolic} can be written in the form of equation~\eqref{main_equation} with initial condition $0$ and hence~\eqref{equation: example_parabolic} has a mild solution $u(t,x,\omega)$ with values in $H_n$ and with c\`adl\`ag trajectories.

    \begin{remark}
        In Hypothesis~\ref{hypothesis: space-time noise parabolic}-(b) one can replace the condition on $g$ by
            \[ g:H_n \to L(W^{-p,2}(\mathcal{D}),H_n) \]
        where $p>\frac{d}{2}$ is a real number, since the embedding $L^2(\mathcal{D})\hookrightarrow W^{-p,2}(\mathcal{D})$ is Hilbert-Schmidt (see Walsh~\cite{Walsh} page 334).
    \end{remark}

	\begin{remark}
		One can use the same arguments as above and the technique used in Example~\ref{example: finite_diemnsional_noise_hyperbolic} to study the second order hyperbolic equation,

    \begin{equation}\label{equation: example_hyperbolic}
        \left\{\begin{array}{lrll}
             \frac{\partial^2}{\partial t^2} u(t,x) = & \Delta u + f(u(t),\frac{\partial u}{\partial t}) + g(u(t^-)) \frac{\partial W}{\partial t} & & \\
             & + k(u(t^-)) \frac{\partial Z}{\partial t}  & \textrm{on} &  [0,\infty) \times \mathcal{D} \\
             u=0 && \textrm{on} &  [0,\infty) \times \partial \mathcal{D} \\
             u(0,x)=0 && \textrm{on} & \mathcal{D} \\
            \frac{\partial u}{\partial t} (0,x)=0 && \textrm{on} & \mathcal{D}.
          \end{array} \right.
    \end{equation}
	where
	    \begin{hypothesis}\label{hypothesis: Second Order Hyperbolic}
        \begin{description}

            \item[(a)] $f:H_{n+1}\times H_n \to H_n$ is measurable, demicontinuous and there exists constants $M$ and $C$ such that for any $u,u_1,u_2 \in H_{n+1}, v,v_1,v_2\in H_n$,
                \[ \langle f(u,v_1)-f(u,v_2),v_1-v_2 \rangle \le M \|v_1-v_2\|^2, \]
                \[ \|f(u_1,v)-f(u_2,v)\|\le C \|u_1-u_2\|. \]

            \item[(b)] $g:H_{n+1} \to L_{HS}(L^2(\mathcal{D}), H_n)$ and $k:H_{n+1} \to L(L^2(\mathcal{D}), H_n)$ are Lipschitz.

            \item[(c)] There exists a constant $D$ such that for $u \in H_{n+1}$, and $v \in H_n$
                \[ \| f(u,v)\|^2 + \| g(u)\|^2 + \|k(u)\|^2 \le D (1+\|u\|^2+\|v\|^2).\]

        \end{description}
    \end{hypothesis}
	
	It follows that under Hypothesis~\ref{hypothesis: Second Order Hyperbolic}, equation~\eqref{equation: example_hyperbolic} has a mild solution $u(t,x,\omega)$ with values in $H_{n+1}$ and with c\`adl\`ag trajectories.
	\end{remark}

\end{example}

\begin{example}[Stochastic Delay Equations]
    Consider the following delay differential equation in $\mathbb{R}$,

    \begin{equation}\label{equation: example_delay_0}
        \left\{\begin{array}{ll}
            dx(t)=&\left( \int_{-1}^0 x(t+\theta) \right)dt - \sqrt[3]{x(t)}dt + x(t) dZ_t \\
            x(\theta)=& \sin(\pi \theta) ,\quad \theta \in (-1,0].
          \end{array} \right.
    \end{equation}
	where $Z_t$ is a real valued square integrable L\'evy process. We apply the results of previous sections and show that this equation has a unique cadlag mild solution. Moreover, $- \sqrt[3]{x}$ can be replaced by any continuous decreasing real function with linear growth. We generalize the above equation as follows:

    \begin{equation}\label{equation: example_delay}
        \left\{\begin{array}{ll}
            dx(t)=&\left( \int_{-h}^0 \mu (d\theta)x(t+\theta) \right)dt + f(x(t))dt + g(x(t))dW_t + k(x(t)) dZ_t \\
            x(\theta)=& \psi(\theta), \theta \in (-h,0].
          \end{array} \right.
    \end{equation}
    where $h>0$, $\mu$ is a measure on $(-h,0]$ with finite variation, $W_t$ is a standard Wiener process in $\mathbb{R}$, $Z_t$ is a pure jump L\'evy martingale in $\mathbb{R}$ and $\psi(\theta)\in L^2((-h,0])$. Moreover assume that,

    \begin{hypothesis}\label{hypothesis: delay}
        \begin{description}

            \item[(a)] $f:\mathbb{R} \to \mathbb{R}$ is continuous and there exists a constant $M$ such that for any $a<b$,
                \[ f(a)-f(b) \le M (a-b),\]

            \item[(b)] $g:\mathbb{R} \to \mathbb{R}$ and $k:\mathbb{R} \to \mathbb{R}$ are Lipschitz.

            \item[(c)] There exists a constant $D$ such that for $a\in \mathbb{R}$,
                \[ | f(a)|^2 + | g(a)|^2 + |k(a)|^2 \le D (1+a^2).\]

        \end{description}
    \end{hypothesis}

    \begin{remark}
        Peszat and Zabczyk~\cite{Peszat-Zabczyk} have studied this delay differential equation with Lipschitz coefficients. We have replaced Lipschitzness of $f$ by the weaker assumption of semimonotonicity.
    \end{remark}

    Let $H=\mathbb{R}\times L^2((-h,0])$ and define the operator $A$ on $H$ by
        \[ A \left( {\begin{array}{c} u \\ v \end{array}} \right) = \left( {\begin{array}{c}  \int_{-h}^0 v(\theta) \mu(d\theta) \\ \frac{\partial v}{\partial \theta} \end{array}} \right). \]
    According to Da Prato and Zabczyk~\cite{DaPrato_Zabczyk_book}, Proposition A.25, the operator $A$ with domain
        \[ D(A)=\left\{ \left( {\begin{array}{c} u \\ v \end{array}} \right) \in H : v\in W^{1,2}(-h,0), v(0)=u \right\} \]
    generates a $C_0$ semigroup $S_t$ on $H$. Let $K=E=\mathbb{R}$ and let $\tilde{N}$ be the compensated Poisson random measure associated with $Z_t$. Define for $\left( {\begin{array}{c} u \\ v \end{array}} \right)\in H$ and $\xi\in \mathbb{R}$,
        \[ \bar{f}(u,v)= \left( {\begin{array}{c} f(u) \\ 0 \end{array}} \right), \bar{g}(u,v)= \left( {\begin{array}{c} g(u) \\ 0 \end{array}} \right), \bar{k}(\xi,u,v)= \left( {\begin{array}{c} \xi k(u) \\ 0 \end{array}} \right).\]
    It is easy to verify that $\bar{f}$, $\bar{g}$ and $\bar{k}$ satisfy Hypothesis~\ref{main_hypothesis}. Now, if we let
        \[X(t)= \left( {\begin{array}{c} x(t)\\ x_t \end{array}} \right) \]
    where $x_t(\theta)=x(t+\theta)$ for $\theta\in (-h,0]$, then equation~\eqref{equation: example_delay} can be written as
        \[dX(t)=A X(t) dt+ \bar{f}(X(t))dt+ \bar{g}(X(t^-))dW_t + \int_E \bar{k}(\xi, X(t^-)) \tilde{N}(dt,d\xi) \]
    with initial condition
        \[ X(0)= \left( {\begin{array}{c} \psi(0) \\ \psi \end{array}} \right) \]
    and hence by Theorem~\ref{theorem:existence and uniqueness} has a unique mild solution $x(t,\omega)$ with c\`adl\`ag trajectories and the solution depends continuously on initial condition.
\end{example}

\subsection*{Acknowledgments}
	The authors wish to thank Professor Carlo Marinelli for introducing them to the valuable article~\cite{Marinelli-Rockner-wellposedness}.

\end{document}